\documentclass[final,nomarks]{dmtcs-episciences}

\usepackage[utf8]{inputenc}
\usepackage{subfigure}

\usepackage[round]{natbib}

\usepackage{amsfonts}
\usepackage{amsmath,bm}
\usepackage{amssymb,mathrsfs}
\usepackage{amsthm}

\newtheorem{theorem}{Theorem} 

\newtheorem{lemma}[theorem]{Lemma}
\newtheorem{corollary}[theorem]{Corollary}
\theoremstyle{definition}
\theoremstyle{definition}\newtheorem{example}[theorem]{Example}
\theoremstyle{definition}

\newcommand{\bB}{\mathbb{B}}
\newcommand{\bL}{\mathbb{L}}

\newcommand{\bN}{\mathbb{N}}

\newcommand{\bR}{\mathbb{R}}

\newcommand{\bG}{\mathbb{G}}
\newcommand{\bS}{\mathbb{S}}
\newcommand{\bV}{\mathbb{\hspace{-.07mm}V\hspace{-.2mm}}}

\newcommand{\cB}{\mathcal{B}}

\newcommand{\cL}{\mathcal{L}}

\newcommand{\fM}{\mathfrak{M}}

\newcommand{\unif}{\text{\rm unif}}
\newcommand{\var}{\text{\rm var}}
\newcommand{\cov}{\text{\rm cov}}

\newcommand{\DST}{\text{\rm DST}}
\newcommand{\BST}{\text{\rm BST}}

\newcommand{\todistr}{\to_{\text{\rm\tiny distr}}} 
\newcommand{\tow}{\to_{\text{\rm\tiny w}}}

\newcommand{\brend}{\hfill $\triangleleft$} 

\newcommand{\oP}{\prec}

\newcommand{\leP}{\preceq}

\allowdisplaybreaks

\title[Limits of sequences of binary trees]
      {A note on limits of sequences of binary trees}
\author{Rudolf Gr{\"u}bel}
\affiliation{
       Leibniz Universit{\"a}t Hannover, Hannover, Germany}

\begin{document}

\publicationdata{vol. 25:1}{2023}{17}{10.46298/dmtcs.10968}{2023-02-16; 2023-02-16; 2023-04-04}{2023-05-11}

\maketitle

\begin{abstract}
We discuss a notion of convergence for binary trees that is based on
subtree sizes. In analogy to recent developments in the theory of graphs, posets and
permutations we investigate some general 
aspects of the topology, such as a characterization of the set of possible 
limits and its structure as a metric space. For random trees the subtree 
size topology arises in the context of algorithms for searching and sorting 
when applied to random input, resulting in a sequence of nested trees. For these
we obtain a structural result based on a local version of exchangeability. 
This in turn leads to a central limit theorem, with possibly mixed asymptotic 
normality.
\end{abstract}

\section{Introduction}\label{sec:intro}
A description of large discrete objects can be based on a suitable convergence
concept, together with a characterization of the possible limits. For
graphs~\cite{LovSze} used subgraph counts and obtained
a description of the limits as \emph{graphons}; see also~\cite{Lovasz} and the references
given there. A similar approach has been used in~\cite{JansonPoset} for posets, 
in~\cite{ElekTardos} for trees
and in~\cite{Hopp} for permutations; in the latter case 
pattern counting leads to \emph{permutons} 
as limit objects. (Some details are given below at the end of Section~\ref{sec:sts}.)
In the present note we use subtree sizes in a similar fashion to obtain a convergence 
concept for binary trees, and we obtain a description of the limit objects as probability
distributions on the set of infinite sequences of zeros and ones. 

In~\cite{LovSze,Hopp,ElekTardos} randomness appears somewhat 
implicitly in the relation to subsampling. It is further used in~\cite{LovSze} and
\cite{Hopp}, via a suitable 
probabilistic construction, to show that each of the potential limit objects indeed 
occurs for some sequence of graphs or permutations. Sequences of binary trees arise in connection
with algorithms for searching and sorting: With random input both the binary search tree (BST) 
and the digital search tree (DST) algorithms~\cite[Chapter 6]{Knuth} lead to increasing
random sequences 
$( X_n)_{n\in\bN}$ of binary trees, where $X_n$ has $n$ nodes 
(again, details are given below).
The asymptotics of such sequences have been studied in~\cite{EGW1} where the subtree size
topology appears in the context of Markov chain boundary theory. A similar boundary theory
interpretation for the topologies of substructure sampling has been found for graph sequences 
in~\cite{GrueDMTCS} and for sequences of permutations in~\cite{Grue22}.

A different class of binary trees appears in connection with 
R\'emy's algorithm, which provides a sequence $(X_n)_{n\in\bN}$ of 
trees where each $X_n$ is uniformly distributed on the set of trees with $n$ nodes. 
This is again a combinatorial Markov chain in the sense of~\cite{GrSemBerKMK}, 
and its
Martin boundary has been determined in~\cite{EGW2}.
In contrast to~\cite{EGW1}, where the boundary was worked out directly through the
Martin kernel, the approach in~\cite{EGW2} is based on the construction of exchangeable 
arrays and an associated representation theorem, as discussed in depth 
in~\cite{KallSym}. This also leads to
a description of  $(X_n)_{n\in\bN}$ as the result of sampling from a real tree in
a specific manner, similar to the use of graphons and permutons. 
For graphons an exchangeability approach is outlined 
in~\cite[Section 11.3.3]{Lovasz} and studied in more detail in~\cite{DiaconisJanson}; 
for a similar treatment of randomly growing permutations see~\cite{Grue22}.
Recently, \cite{ElekTardos} constructed \emph{dendrons} as limit objects for general trees,
with an approach based on regarding trees as metric spaces, together with a suitable 
rescaling, and using the machinery of ultraproducts and ultralimits. Finally, rooted general
trees and their limits appear in connection with classical branching processes; 
see the survey~\cite{JansonSurvey} and the references given there.

It is well known that search trees and uniform  trees belong to two different
`universality classes', often labeled by the asymptotics 
of their height, which is $\log n$ in the first and $\sqrt{n}\,$ in the second case.
Another aim of this note is to show that an approach based on probabilistic symmetries
can also be used in the context of trees of logarithmic height.

In Section~\ref{sec:sts} we first introduce some basic notation for binary trees and
then study the subtree size topology, proceeding essentially 
as in~\cite{Hopp} for permutation sequences. In Section~\ref{sec:symm} we consider 
tree sequences that grow by one node at a time, such as the output sequences obtained
with the BST and DST algorithms mentioned above, where we  introduce a local notion of
exchangeability. This is then applied in Section~\ref{sec:FCLT} 
to obtain a second order result for subtree size convergence, where a possibly mixed Gaussian process arises as the distributional limit.

We restrict ourselves to  binary trees in order to arrive at a compact presentation. 
Many related varieties of trees, such as quad trees, may be treated in a similar manner;
see also the models considered in~\cite{DevSplit} and in~\cite{EGW1}. 

\section{Subtree size convergence}\label{sec:sts}

Let $\bV:=\{0,1\}^\star:= \bigsqcup_{k=0}^\infty \{0,1\}^k$ be the set of finite words 
with letters from 
the alphabet $\{0,1\}$. We write 
$|u|=k$ for the \emph{length} of the word $u=(u_1,\ldots,u_k)$ and $\bV_k$ for the set of words of
length $k$. We will also use the notation $|A|$ for the size of a set $A$. The \emph{concatenation} of $u=(u_1,\ldots,u_k)$ and $v=(v_1,\ldots,v_l)$ is
given by $u+v=(u_1,\ldots,u_k,v_1,\ldots,v_l)$. On $\bV$, the
\emph{prefix order} is defined by $v=(v_1,\dots,v_k)\oP w=(w_1,\ldots,w_l)$ if
$k<l$ and $v_i=w_i$ for all $i\in[k]:=\{1,\ldots,k\}$. As usual, we augment this by 
putting $v\leP w$ if $v\oP w$ or $v=w$. 

By a \emph{binary tree} $x$ we mean a subset of 
$\bV\,$  (the potential \emph{nodes} or \emph{vertices} of the tree) 
with the property that $v=(v_1,\ldots v_k)\in x$ with $k>0$ implies
$(v_1,\ldots,v_{k-1})\in x$. In short, binary trees are sets of words that are 
\emph{prefix stable}. 
The node $v=\emptyset$ (arising if $k=0$) is the \emph{root} of the tree.
Further, $v1:=(v_1,\ldots,v_k,1)$ and $v0:=(v_1,\ldots,v_k,0)$ are the
\emph{right} and \emph{left descendant} of $v$ respectively. The set $\bV$ of all nodes
may be seen as the complete infinite binary tree, $\bB_n$ is the set of binary trees
with $n$ nodes, and $\bB:=\bigsqcup_{n=0}^\infty \bB_n$ is the set of all binary trees
with finitely many nodes. The (external) \emph{boundary} $\partial x$ of a finite tree
consists of all \emph{external nodes} $v\in \bV\setminus x$ with $v=u0$ or $v=u1$ 
for some $u\in x$.
It is easy to see that $|\partial x|=|x|+1$ for all $x\in\bB$.
 
Two tree-related notions that are particularly important for us are the \emph{subtree}
$\sigma(x,u)$ of a tree $x$ rooted at $u\in x$ and the (relative) \emph{subtree size function}
$t(x,\cdot):\bV \to [0,1]$ of $x$, with
\begin{equation}\label{eq:defsts}
     \sigma(x,u) := \{v\in \bV: \, u+v\in x\},\quad
     t(x,u) := \frac{1}{|x|} \,|\sigma(x,u)|\quad\text{for all } u\in \bV.
\end{equation}
We say that a sequence $(x_n)_{n\in\bN}$ \emph{converges in the subtree size topology}\,
if, for all $u\in\bV$, the real numbers $t(x_n,u)$ converge as $n\to\infty$. It is easy to
see that finite binary trees are characterized by their subtree size function. We may therefore regard
the mapping $\bB\ni x\mapsto (u\mapsto t(x,u))\in [0,1]^\bV$ as an embedding of the set of
finite binary trees into a set that is compact by Tychonoff's theorem under the topology of pointwise convergence, as in the definition of subtree size convergence, and may even 
identify trees with their subtree size functions. In this sense the closure $\overline\bB$ 
of the image of the embedding provides a compactification where the limits 
are given by the functions on $\bV$ that appear as pointwise limits of the sequences $(t(x_n,\cdot))_{n\in\bN}$ for convergent sequences of trees.
Obviously, not all functions on $\bV$ can arise in this way, and the identification of subtree size limits amounts to finding a tractable space that is homeomorphic to the 
boundary $\overline \bB\setminus \bB$. Note, however,  that the general 
abstract setting immediately yields that each sequence of trees has a convergent 
subsequence.

Let  $\bV_\infty:=\{0,1\}^\infty$ be the set of infinite sequences of zeros and ones and let
$\cB_\infty$ be the $\sigma$-field on the sequence space that is generated by the 
coordinate projections $(v_n)_{n\in\bN}\mapsto v_k$,
$k\in\bN$. Let $\fM_\infty:=\fM(\bV_\infty,\cB_\infty)$ be the set of 
probability measures on $(\bV_\infty,\cB_\infty)$. 
We use the canonical extension of the prefix order to pairs $(u,v)$ with 
$u\in \bV$ and $v\in\bV_\infty$. 
For all $u\in \bV$ let $B_u:=\{v\in \bV_\infty: \, u\oP v \}$. Then 
\begin{equation}\label{eq:defB}
 B_u\cap B_v=\begin{cases}
          B_v,&\text{if } u\preceq v,\\ B_u,&\text{if } v\preceq u,\\ \emptyset, &\text{otherwise,}
             \end{cases}
\end{equation}
which implies that $\cB_0:=\{B_u:\, u\in \bV\}$ is a countable and intersection stable generator of~$\cB_\infty$. As a consequence,  elements of $\fM_\infty$ are determined 
by their values on $\cB_0$. 
With componentwise addition $\bV_\infty$ becomes a compact group. Its unique Haar measure $\mu$ 
with total mass 1, the uniform distribution on $(\bV_\infty,\cB_\infty)$, is characterized by
$\mu(B_u)=2^{-|u|}$ for all $u\in\bV$.

We will need the following measure-theoretic property of binary trees, which
seems to be part of the folklore of the subject. I have not found a suitable 
reference, and therefore include a proof.

\begin{lemma}\label{lem:cara}
Let $\psi:\bV\to [0,1]$ be such that $\psi(\emptyset)=1$ and
\begin{equation}\label{eq:finadd}
   \psi(u)=\psi(u0)+\psi(u1)\quad\text{for all } u\in\bV.
\end{equation}
Then there exists a unique $\mu\in\fM_\infty$ such that 
$\mu(B_u)=\psi(u)$ for all $u\in\bV$. 
\end{lemma}

\begin{proof}
We define a set function $\mu_0:\cB_0\to [0,1]$ by $\mu_0(B_u)=\psi(u)$ 
for all $u\in\bV$. Using~\eqref{eq:finadd} 
it is easy to show by induction that $\mu_0$ is finitely additive on each 
system $\{B_u:\, u\in \bV_k\}$, $k\in\bN$. 
Suppose now that $B_{u(1)},\ldots,B_{u(n)}\in\cB_0$
are pairwise disjoint and let $k:=\max_{i\in [k]}|u(i)|$.  We then get
\begin{align*}
 \mu_0\Bigl(\sum_{i\in [n]} B_{u(i)}\Bigr) \ =
            \ \sum_{i\in [n]}\ \sum_{v\in\bV_k,u(i)\preceq v} \mu(B_v) \ = 
            \ \sum_{i\in [n]}  \mu(B_{u(i)}). 
\end{align*}
The finite additivity of $\mu_0$ on $\cB_0$  extends to the field $\cB_1$ 
generated by $\cB_0$.
 
We now use a topological argument:
An ultrametric $d$ can be defined on $\bV_\infty$ by $d(v,w) = 2^{-|v\wedge w|}$,
where $|v\wedge w|$ denotes the length of the 
longest common prefix of the sequences $v,w$. 
Endowed with $d$ the sequence space becomes a totally 
disconnected and compact topological space, with $\cB_\infty$ as its Borel $\sigma$-field.
The $\sigma$-additivity of $\mu_0$ on $\cB_1$ now follows from the finite intersection 
property of compact sets, so that we may apply Carath{\'e}odory's 
extension theorem.
\end{proof}

The lemma shows that $\fM_\infty$ can be embedded into $[0,1]^\bV$, as done 
above for $\bB$. In its proof we have chosen a topological argument, in line with 
the general thrust of the paper; Kolmogorov's consistency theorem can be used to obtain   
a probabilistic alternative.

The digital search tree (DST) algorithm turns a sequence $(\xi_i)_{i\in\bN}$
of elements of $\bV_\infty$ into an increasing sequence $(x_n)_{n\in\bN}$ of binary trees, 
with $x_n\in\bB_n$ for all $n\in \bN$: Starting with $x_1=\{\emptyset\}$ we obtain 
$x_{n+1}$ from $x_n$ and $\xi_n$ by interpreting $\xi_{n}$ as a routing
instruction, with 0 as a move to the left and 1 as a move to the right,
and the inclusion of the external node $u$ where exit from the current tree $x_n$ 
occurs. For $\mu\in\fM_\infty$ let $\DST(\mu)$ be the distribution of the $\bB$-valued 
random sequence $(X_n)_{n\in\bN}$ generated by the digital search tree algorithm
if the input sequence $(\xi_i)_{i\in\bN}$ consists of independent random variables with 
distribution $\mu$. For example, if 
$\mu$ is concentrated at the single sequence $(0,0,0,\ldots)\in\bV_\infty$ then
the DST mechanism produces the infinite tree that consists of all nodes on
the left-most infinite branch in $\bV$. A special case of the DST family is the
\emph{Bernoulli model} with parameter $p\in (0,1)$, see~\cite[Section 1.4.3]{Drmota}, 
where each $\xi_i$ consists of a sequence of independent $\{0,1\}$-valued 
variables $(\xi_{ik})_{k\in\bN}$ with $P(\xi_{ik}=1)=p$ for all $k\in\bN$.  Especially
the symmetric case, with $p=1/2$, has been studied extensively. 

It is easy to see that for a sequence $(x_n)_{n\in\bN}$ of trees with 
$\liminf_{n\to\infty}|x_n|<\infty$ subtree size convergence implies that the sequence is 
constant from some $n_0\in\bN$ onwards. The following may be regarded as the binary
tree analogue of~\cite[Theorem 1.6]{Hopp}. 

\begin{theorem}\label{thm:sts}

\noindent
\emph{(a)} If a sequence $(x_n)_{n\in\bN}$ of binary trees with $\lim_{n\to\infty} |x_n|=\infty$
converges in the subtree size topology then, for some unique $\mu\in\fM_\infty$,
\begin{equation}\label{eq:limsts}
     \lim_{n\to\infty} t(x_n,u) = \mu(B_u)\quad \text{for all }u\in\bV.
\end{equation}

\noindent
\emph{(b)} Let $\mu\in\fM_\infty$ and let $(X_n)_{n\in\bN}$ be distributed according to 
$\DST(\mu)$. Then $X_n$ converges with probability one to $\mu$ in the subtree size topology.
\end{theorem}

\begin{proof} (a) Let $\psi(u):= \lim_{n\to\infty} t(x_n,u)$ for all $u\in\bV$. 
It is easy to see that $\psi$ satisfies~\eqref{eq:finadd} and clearly,
$\psi(\emptyset)=1$. Lemma~\ref{lem:cara} now supplies the probability measure $\mu$
and, by~\eqref{eq:limsts}, the tree sequence converges to $\mu$ in the 
subtree size topology.

\smallbreak
(b) Let $u=(u_1,\ldots,u_k)\in\bV$. We may assume that $\mu(B_u)>0$ since 
$t(X_n,u)=0$ with probability one for all $n\in\bN$ otherwise, which would imply~\eqref{eq:limsts}.
The entry times $\tau_u:=\inf\{n\in\bN:\, u\in X_n\}$ can be written as
\begin{equation*}
   \tau_u \;=\; \tau_{(u_1)} \,+\, \bigl(\tau_{(u_1,u_2)}-\tau_{(u_1)}\bigr) \,+ \cdots +\,
                         \bigl(\tau_{(u_1,u_2,\ldots,u_k)}-\tau_{(u_1,\ldots,u_{k-1})}\bigr).
\end{equation*}
It follows from the description of the DST algorithm that, with $\tau_\emptyset:=0$, 
the differences
\begin{equation*}
       \tau_{(u_1,u_2,\ldots,u_i)}-\tau_{(u_1,u_2,\ldots,u_{i-1})},\quad i=1,\ldots,k,
\end{equation*}
are independent and geometrically distributed with (success) parameter
$\mu(B_{(u_1,\ldots,u_i)})$. Hence $\tau_u<\infty$ with probability one. From $n=\tau_u(\omega)$
onwards an increase in the size of the subtree rooted at $u$ is equivalent 
to $\xi_n(\omega)\in B_u$. These are independent events with the same positive probability
$\mu(B_u)$, which leads to~\eqref{eq:limsts}.
\end{proof}

The first part of the theorem shows that convergence with respect to the subtree size topology leads to convergence of binary trees to a measure in $\fM_\infty$, and 
the second part shows that indeed  each $\mu\in\fM_\infty$ 
arises as the limit of a sequence of binary trees. Using the above identifications of $\bB$, 
$\overline\bB$ and $\fM_\infty$ as subsets of $[0,1]^\bV$ we may summarize the result 
by the simple formula $\overline\bB=\bB\sqcup \fM_\infty$. 

\begin{example}\label{ex:sts1}
Let $x_n:=\bigsqcup_{k=0}^n \{0,1\}^k$ be the complete finite binary tree of height $n$.
Then, for $k\le n$ the subtree of $x_n$ rooted at $u=(u_1,\ldots,u_k)\in\bV$ is
isomorphic to $x_l$ with $l=n-k$, which leads to $\lim_{n\to\infty} t(x_n,u)= 2^{-k}$. It follows
that $x_n\to\mu$, with $\mu$ the uniform distribution on $\bV_\infty$.
\brend
\end{example}

Similar to the graphon and permuton situation, the limit of a random sequence of
binary trees may be a (truly) random element of $\fM_\infty$. The next example  
has already been
considered in~\cite{EGW1}, with methods from Markov chain boundary theory.

\begin{example}\label{ex:sts3} (BST, see also~\cite[Example 1]{DevSplit})
Let $(\xi_i)_{i\in\bN}$ be a sequence of independent random variables, all uniformly 
distributed on the unit interval. We may assume that the values are pairwise different, and
may then define a random sequence $(R_n)_{n\in\bN}$ by 
$R_n=\bigl|\{i\in[n]:\, \xi_i\le \xi_n\}\bigr|$.
As in the DST case, the BST algorithm generates a sequence $(X_n)_{n\in\bN}$ of increasing
trees, with $X_1=\{\emptyset\}$ and $X_{n+1}=X_n\sqcup \{v\}$ with some 
$v\in\partial X_n$. 
To specify the respective new node
as a function of $X_n$ and $R_{n+1}$ we first note that the $n+1$ elements of $\partial X_n$
can be ordered lexicographically, 
and we then take the node $v\in\partial X_n$  with left-right position $R_{n+1}$. 
We write $X=(X_n)_{n\in\bN}\sim \BST$ for the result. In a nutshell, 
BST uses the ranks whereas DST uses the bit structure of the input
values. This implies that we may replace $\unif(0,1)$ by any other distribution $\mu$ 
as long as $\mu(\{a\})=0$ for all $a\in\bR$. With this construction all $\xi$-values less than
$\xi_1$ end up in the left subtree, the larger ones in the right subtree of the root node. 
It follows that 
$t(X_n,(0))$ converges almost surely (a.s.) to $\xi_1$ and $t(X_n,(1))$ to $1-\xi_1$. Further, given
$\xi_1=a$ the values less than $a$ and greater than $a$  are independent and uniformly
distributed on $[0,a)$ respectively $ (a,1]$.
Hence, given $\xi_1$, the left and right subtree are independent and, after passing to the
appropriate subsequence, equal in distribution to $X$. Taken together this shows that for
any $u\in\bV$, the sequence of pairs 
$\bigl(t(X_n,u0)/t(X_n,u),t(X_n,u1)/t(X_n,u)\bigr)$ converges
almost surely to $(\eta_u,1-\eta_u)$, where $\eta_u$, $u\in\bV$, are independent and uniformly
distributed on $[0,1]$. Thus, the BST sequence converges almost surely to a random element
$M_{\text{\tiny BST}}$ of $\fM_\infty$, with 
\begin{equation}\label{eq:Mbst}
     M_{\text{\tiny BST}}(B_u)\, =\, 
                \prod_{i=1}^{k} \eta_{(u_1,\ldots,u_i)}^{1-u_i}
                    \bigl(1- \eta_{(u_1,\ldots,u_i)}\bigr)^{u_i}, 
                              \quad u=(u_1,\ldots,u_k)\in\bV.
\end{equation}
In particular, $P(M_{\text{\tiny BST}}=\mu)=0$ for all $\mu\in\fM_\infty$. 
\brend
\end{example}

Our final example requires a slight shift of perspective, from random variables to
their distributions. We use the classic~\cite{Bill} as our basic reference for
weak convergence.

By Prohorov's theorem~\cite[Theorems 6.1 and 6.2]{Bill}, the space $\fM_1(\overline\bB)$ 
of probability measures on (the Borel subsets of) $\overline \bB$, together with the 
topology of weak convergence, is a compact metrizable space. We write temporarily 
$\tilde M_n,\tilde M$ for (non-random) elements of $\fM_1(\overline\bB)$ in order to
distinguish these from random elements $M_n,M$ of $\,\overline\bB$ (thus, we may have 
$\tilde M=\cL(M)$). The topological structure implies that any sequence 
$(\tilde M_n)_{n\in\bN}$ must have a limit point in $\fM_1(\overline\bB)$, and 
convergence on this level, which we denote by $\tilde M_n\tow \tilde M$, 
holds if and only if there is only one such point. 

In contrast to the previous example, where the model specifies the distribution 
of the full sequence $(X_n)_{n\in\bN}$, we now only have the  distributions
of the individual variables $X_n$, $n\in\bN$. In view of its connection to enumerative
combinatorics the uniform distribution is of special interest.

\begin{example}\label{ex:sts2}
Let $\tilde M_n=\unif(\bB_n)$ for all $n\in\bN$, and let $\tilde M$ be an associated 
limit point, so that $\tilde M_{n(k)}\tow \tilde M$ as $k\to\infty$ for some 
subsequence $(n(k))_{k\in\bN}$. Then, by the Skorohod representation theorem, 
see e.g.~\cite[Theorem 3.30]{KallFMP}, there exists a probability space
carrying random variables $X_\infty,X_1,X_2,\ldots$ with $\cL(X_\infty)=\tilde M$,
$\cL(X_k)=\tilde M_{n(k)}$ for all $k\in\bN$, such that $X_k\to X_\infty$ 
a.s.\ in the subtree size topology. We will show that 
\begin{equation}\label{eq:0-1}
        P\bigl(X_\infty(B_u)\in\{0,1\}\bigr)\,=\,1\quad\text{for all }u\in\bV. 
\end{equation}
For this, we note that any $x\in\bB$ may be decomposed into its left and right subtree,
given by $\sigma(x,(0))$ and $\sigma(x,(1))$ respectively. Further, for all $n\in\bN$,
\begin{equation*}
    |\bB_n|= C_n:=\frac{1}{n+1}\binom{2n}{n},
\end{equation*}
one of the many appearances of the Catalan numbers $C_n$. Hence, if $U_n\sim\unif(\bB_n)$,
and with $L_n:= |\sigma(U_n,(0))|$, $R_n:=|\sigma(U_n,(1))|$,
\begin{equation*}
       P(L_n=k)=P(R_n=k) = \frac{C_k\, C_{n-1-k}}{C_n}\quad\text{for }k=0,\ldots,n-1.
\end{equation*} 
Standard bounds for the Catalan numbers lead to
\begin{equation*}
    \lim_{n\to\infty} P(an<L_n<bn) = \lim_{n\to\infty} P(an<R_n<bn) \,= 0 
                 \quad\text{for all }  0<a<b<1,
\end{equation*}
and it follows that $\cL\bigl(t(U_n,(0))\bigr)=\cL(L_n/n)$ converges weakly to the
uniform distribution on the finite set $\{0,1\}$. For the representing sequence 
$(X_{k})_{k\in\bN}$ we must have almost sure convergence of $t(X_{k},(0))$ 
to some real value $X_\infty(B_{(0)})$, hence~\eqref{eq:0-1} holds for $u=(0)$ and $u=(1)$. 
Uniformity of the distribution further implies that, conditionally on 
$L_n=k$, the left and right subtree of $U_n$ are independent and uniformly 
distributed on $\bB_k$ and $\bB_{n-1-k}$ respectively. Applying the above argument
to these we obtain~\eqref{eq:0-1} for nodes of length two, and iteration gives
the statement for all $u\in\bV$.

If $\mu\in\fM_\infty$ is such that $\mu(B_u)\in\{0,1\}$ for all $u\in\bV$ then  
$\mu=\delta_v$ for some $v\in\bV_\infty$. The limit point 
$\tilde M\in\fM_1(\overline\bB)$ is therefore
concentrated on the subset $\{\delta_v:\, v\in\bV_\infty\}$ of $\fM_1(\overline\bB)$. 
We next apply a symmetry 
argument: The group $\bV_\infty$ acts on $\bV$ via
\begin{equation*}
    v. u := (w_1,\ldots,w_k), \quad \text{with }\ w_j:=v_j+u_j \bmod 2,\ j=1,\ldots,k,
\end{equation*}
where $v=(v_j)_{j\in\bN}\in\bV_\infty$, $u=(u_1,\ldots,u_k)\in\bV$ and $k=|u|$. This
preserves prefix order, hence $\bV_\infty$ also acts on $\bB$ via
\begin{equation*}
   v.x :=\{v.u:\, u\in x\}\quad \text{for all } v\in\bV_\infty,\,x\in \bB.
\end{equation*}
Clearly, if $U_n\sim\unif(\bB_n)$, then $v.U_n\sim\unif(\bB_n)$. Taken together 
this shows the distribution of $\tilde M$ is invariant under these transformations, which 
implies that $\tilde M=\tilde M_{\text{\tiny unif}}:=  \cL(\delta_V)$, with $V$ uniformly distributed on $\bV_\infty$.

Thus, all limit points are identical, and we have 
$\unif(\bB_n)\tow \tilde M_{\text{\tiny unif}}$ as $n\to\infty$.
\brend
\end{example}

It follows from this example that for any sequence $(X_n)_{n\in\bN}$ of random 
trees on some probability 
space  with the properties that $\cL(X_n)=\unif(\bB_n)$ for all $n\in\bN$ and 
that $X_n$ converges almost surely to some $X_\infty$ in the subtree size topology, 
we must have $\cL(X_\infty)=\cL(\delta_V)$ with $\cL(V)=\unif(\bV_\infty)$.
R\'emy's algorithm, see~\cite{Remy},  provides such a sequence. 
In ~\cite{EGW2} a different topology 
has been introduced and discussed for the R\'emy sequence, 
and this led to a more detailed class of limits.
Stated somewhat informally, subtree sizes reflect the local behavior,
and in the uniform case this amounts to a reduction of the limit tree to its 
\emph{spine}, a term commonly used in connection with the asymptotics of 
Galton-Watson trees; see~\cite[p115]{JansonSurvey}. Moreover, for uniform binary trees 
the spine can be constructed from a sequence of coin tosses. 

Another opportunity 
for comparison between topologies arises if we ignore
the root and the left-right positioning of the descendants in a binary tree, so that
we arrive at  an isomorphism class of tree graphs. For these, a `global' 
topology is introduced 
and discussed in~\cite{ElekTardos} and~\cite{JansonTrees}. With the complete binary trees 
in Example~\ref{ex:sts1} the situation turns out to be somewhat reversed as,  for these,
the subtree size topology leads to an arguably more interesting limit;
see~\cite[Example~7.3]{JansonTrees}.

We next investigate the topological structure of subtree size convergence. 
In the general setup, with 
convergence meaning the pointwise convergence of the functions $t(x,\cdot)$, a suitable
metric can be obtained as
\begin{equation*}
       d_w(x,y):=\sum_{u\in\bV} w(u)\bigl|t(x,u)-t(y,u)\bigr|, \quad x,y\in\bB,
\end{equation*} 
with an arbitrary $w:\bV\to(0,\infty)$ such that $\sum_{u\in\bV}w(u)<\infty$. 
However, in this generality
this does not reflect the specific structures considered here. For sequences of graphs and 
permutations embeddings of the discrete structures into the 
respective limit spaces of graphons and permutons have been given in~\cite[Section 1.5.2]{Lovasz} 
and~\cite[Definition 3.4]{Hopp}. To obtain a similar embedding of $\bB$ into $\fM_\infty$ we first
recall the metric $d$ 
from the proof of Theorem~\ref{thm:sts} that makes $\bV_\infty$ a compact ultrametric space.
The $\sigma$-field $\cB_\infty$ on $\bV_\infty$ is the associated Borel $\sigma$-field,
and weak convergence $\mu_n\to\mu$ in $\fM_\infty$ means that $\int f\,d\mu_n\to\int f\,d\mu$ 
for all bounded continuous $f:\bV_\infty\to\bR$. 
We now associate with $x\in\bB$ an element $\mu_x\in\fM_\infty$ by
\begin{equation}\label{eq:embed}
 \mu_x = \frac{1}{|x|+1}\sum_{v\in\partial x}\unif(B_v).
\end{equation}
Here, for $v=(v_1,\ldots,v_k)\in\bV$, the probability measure $\unif(B_v)$ is the 
distribution of the sequence
$(v_1,\ldots,v_k,\xi_1,\xi_2,\xi_3,\ldots)\in\bV_\infty$, where $\xi_i$, $i\in\bN$, are 
independent and uniformly distributed on the set $\{0,1\}$. 

\begin{theorem}\label{thm:ststop} Let $(x_n)_{n\in\bN}$ be a sequence of binary trees with $\lim_{n\to\infty} |x_n|=\infty$.
Then  $(x_n)_{n\in\bN}$ converges in the subtree size topology if and only if
the associated sequence $(\mu_{x_n})_{n\in\bN}$ 
of elements of $\fM_\infty$ defined in~\eqref{eq:embed} 
converges in the weak topology, and then the limits are the same.
\end{theorem}

\begin{proof}
The path through $x\in\bB$ defined by $v\in\bV_\infty$ leaves $x$ at some unique 
$u\in\partial x$. Hence, in view of~\eqref{eq:defB}, the set system $\{B_u:\, u\in\partial x\}$ 
is a measurable partition of $\bV_\infty$. If the subtree of $x$ rooted at $u\in x$ has $k$ nodes, then $|\{v\in\partial x:\, u\prec v\}| = k+1$, so that
\begin{equation}\label{eq:rel-mu-t}
\mu_x(B_u) \ =\  \frac{1}{|x| + 1}\sum_{v\in\partial x, u\prec v}1
           \ =\  \frac{1 + |x|\,t(x,u)}{1+|x|}
\end{equation}
which implies the general bounds 
\begin{equation}\label{eq:bounds}
     0\; \le \; \mu_x(B_u)-t(x,u)\; \le \; \frac{1}{1+|x|}\quad\text{for all }u\in\bV.
\end{equation}
It follows that, for any sequence $(x_n)_{n\in\bN}\subset \bB$ with $|x_n|\to \infty$,
subtree size convergence is
equivalent to the convergence of $\mu_{x_n}(B_u)$ as $n\to\infty$ for all $u\in\bV$. By the Portmanteau theorem~\cite[Theorem 2.1]{Bill}, as each $B_u$ is
open and closed in the compact ultrametric space $(\bV_\infty,d)$, weak convergence of a
sequence $(\mu_n)_{n\in\bN}\subset \fM_\infty$ implies convergence of $\mu_n(B_u)$ as $n\to\infty$
for all $u\in\bV$.
Thus it remains to show that $\{B_u:\, u\in\bV\}$ is a convergence determining
class, but this follows easily with the criteria given in~\cite[p14f]{Bill}.
\end{proof}

We now compare the above to the graph situation, see~\cite{LovSze,Lovasz,DiaconisJanson}, 
and the permutation situation, see~\cite{Hopp}. 
Similar treatments of partially ordered 
sets (posets) and general trees have been given in~\cite{JansonPoset}, and 
in~\cite{ElekTardos} and~\cite{JansonTrees} respectively.

Let $\bG_n$ be the set of simple graphs with $[n]$ as its set of vertices. 
For $G\in\bG_n$ and  $H\in\bG_k$, $k \le n$, let $T(G,H)$ be
the number of injections $\phi:[k]\to[n]$ with the property that, for all 
$1\le j < l\le k$, $\{j,l\}$ is an edge in $H$ if and only if 
$\{\phi(j),\phi(l)\}$ is an edge in $G$. Similarly, with 
$\bS_n$ the set of permutations of $[n]$ and $\pi\in\bS_n$, $\tau\in\bS_k$,
$k\le n$, let $T(\pi,\tau)$ be the number of strictly increasing 
functions $\phi:[k]\to[n]$ with the property that, for all $1\le j < l\le k$, 
$\tau(j) < \tau(l)$ holds if and only if $\pi(\phi(j))<\pi(\phi(l))$. 
Dividing by the respective number of functions $\phi$ leads to subgraph frequencies
$t(G,H)$ and pattern frequencies $t(\pi,\tau)$, and convergence 
of a sequence $(G_n)_{n\in\bN}$
of graphs or $(\pi_n)_{n\in\bN}$ of permutations may be defined as the convergence 
of all substructure frequencies $H\mapsto t(G_n,H)$, 
respectively $\tau\mapsto t(\pi_n,\tau)$.
The associated limit objects are graphons and permutons: A graphon is a symmetric and measurable function 
$W:[0,1]^2\to [0,1]$, and the
analogue of Theorem~\ref{thm:sts}\,(b) consists in defining an isomorphism class $X_n$
of graphs with vertex set $[n]$ by choosing $U_1,\ldots,U_n$ uniformly at random from the
unit interval and then connecting vertices $i$ and $j$ with probability $W(U_i,U_j)$, independently
for $1\le i < j\le n$. A permuton is a distribution function $C:[0,1]^2\to[0,1]$ of a 
distribution with uniform marginals (hence a two-dimensional copula) and 
the analogue of Theorem~\ref{thm:sts}\,(b) 
is based on constructing a random permutation $X_n$ of $[n]$ via the rank plot of independent
random vectors $(Y_i,Z_i)$, $i\in [n]$, with distribution function $C$.

For the binary trees considered here, the role of subgraph respectively pattern
is taken over by a node $u\in \bV$, and instead of substructures we use the prefix
relation: $T(x,u)$ is now the number of nodes $v\in x$ with $u\preceq v$,
and standardization means that we divide by~$|x|$.
All three cases have an obvious sampling interpretation. For a permutation $\pi\in\bS_n$, 
for example, we select a strictly increasing function $\phi:[k]\to[n]$ uniformly
at random from the $\binom{n}{k}$ possibilities, and $t(\pi,\tau)$ emerges
as the probability that the random choice leads to pattern containment. For
binary trees $x\in\bB_n$ we select a node $v$ of $x$ uniformly at random,
and $t(x,u)$ is the probability that $u$ is a prefix of the chosen node.
All three modes of convergence are thus connected to a view according to which 
two large discrete structures of the same type are close to each other if they
appear to be similar when viewed through the `sampling lens'. As in the 
permuton case, we obtain a description of the limit space as the space of all probability 
measures on some compact metric space, with the topology of weak convergence of distributions.
With the Prohorov metric~\cite[p237f]{Bill} this is again a compact metric space. 

Another parallel is the use of accompanying sequences, corresponding to the transition from $x$ 
to $\mu_x$ in the space of limits, together with a result such as Theorem~\ref{thm:ststop} 
relating the convergence of the sequence of interest to the associated sequence in the limit space; see e.g.~\cite[Theorem 1.8]{Hopp} for permutations. Example~\ref{ex:sts1} can be used to show 
that $x\mapsto \mu_x$ is not one-to-one, in contrast to the permutations case, but in analogy to
the poset and graph situation. The basis for such results are equations such 
as~\eqref{eq:rel-mu-t}. For graphs and posets different versions of the substructure sampling are discussed in the literature. 
For trees, given the deterministic relation between number of nodes and size
of the external boundary, we could have worked with    
\begin{equation*}
     \sigma_0(x,u) := \{v\in \bV: \, u+v\in \partial x\},\quad
     t_0(x,u) := \frac{1}{|\partial x|} \,|\sigma_0(x,u)|,
\end{equation*}
which would lead to the more concise the version 
$\mu_x(B_u)=t_0(x,u)$ of~\eqref{eq:rel-mu-t}.

\section{Local exchangeability}\label{sec:symm}
The topological approach of the previous section applies to arbitrary sequences $(x_n)_{n\in\bN}$ of elements of $\bB$. In the present section we assume that
$x_n\in\bB_n$ and $x_n\subset x_{n+1}$ for all $n\in\bN$. 
Such sequences that grow by one node at a time appear in connection with the
$\DST$ and $\BST$ algorithms, for example. Also, the boundary theory approach in~\cite{EGW1} 
refers to random sequences $X=(X_n)_{n\in\bN}$ with these properties, where it is
further assumed that the stochastic process $X$ has the Markov property. 

We assume that $X=(X_n)_{n\in\bN}$ satisfies $P(\bB_\uparrow)=1$, with the
path space defined by
\begin{equation}\label{eq:pathspace}
  \bB_\uparrow :=\Bigl\{ x=(x_n)_{n\in\bN}:\, x_n\in\bB_n,\, x_n\subset x_{n+1}
             \text{ for all }n\in\bN,\; \bigcup_{n=1}^\infty x_n=\bV\,\Bigr\}.
\end{equation}
We endow $\bB_\uparrow$ with the $\sigma$-field $\cB_\uparrow$ 
generated by the coordinate 
projections and write $\fM_\uparrow$ for the set of probability measures on 
$(\bB_\uparrow,\cB_\uparrow)$.  As $\,\bigcup_{n\in\bN} X_n=\bV\,$ with probability one,
we have  $P(\tau_u<\infty)=1$ for all entry times 
$\tau_u:=\inf\{n\in\bN:\, u\in X_n\}$, $u\in\bV$. Ignoring a set of probability zero,  
we may define the \emph{local increment process at} $u$ by
$Y(u)=(Y_n(u))_{n\in\bN}$ by
\begin{equation}\label{eq:locinc}
        Y_n(u)=\begin{cases} 1, &s(\tau_{u}+n,u1)>s(\tau_{u}+n-1,u1),\\
                            -1, &s(\tau_{u}+n,u0)>s(\tau_{u}+n-1,u0),\\  
                             0, &\text{else},  
               \end{cases}  
\end{equation}
where $s(n,u):=|\{v\in\bV:\, u+v\in X_n\}|$. Thus the value of $Y_n(u)$ indicates 
if the right or left subtree of $u$, or none of them,
receives another node at time $\tau_u+n$. We note for later use that the transition
from $X$ to $Y(u)$ may be seen as the result of a deterministic function, 
say $\Psi_u$, defined on $\bB_\uparrow$ and with values in $\{-1,0,1\}^\bN$.

In connection with the representation part of the following theorem we
recall that a statement on conditional distributions such as $\cL(X|Y=y)= Q(y,\cdot)$ 
means that $Q$ is a probability kernel and that, for a class
$A$ of measurable sets sufficiently rich to characterize the distribution 
of~$X$, it holds that $P(X\in A)=\int Q(y,A) \,\cL(Y)(dy)$. In order to be able to
formalize this in the present context, where the values of $X$ and $Y$ are 
distributions, we need a measurable structure on $\fM_\uparrow$. 
As in the case of $\fM_\infty$ we use the $\sigma$-field generated by the 
insertion functions $\mu\mapsto \mu(A)$. 
Finally, we say that an element $\mu$ of $\fM_\infty$ has
\emph{full support} if its support is equal to the whole of $\bV_\infty$. It is easy to
see that this is equivalent to the condition that $\mu(B_u)>0$ for all $u\in\bV$.

\begin{theorem}\label{thm:charSTS}
Suppose that $X=(X_n)_{n\in\bN}$ is such that $P(X\in\bB_\uparrow)=1$.

\vspace{.5mm}
\emph{(a)} If $X$ is \emph{locally exchangeable} in the sense that all
local increment processes $Y(u)$, $u\in\bV$, are exchangeable, then 
there exists a possibly random $M\in \fM_\infty$ such that
\begin{equation}\label{eq:reprSTS}
           \cL(X |M=\mu)\, =\, \DST(\mu)
                 \quad\text{for }\cL(M)\text{-almost all }\mu\in\fM_\infty,
\end{equation}
and $X_n$ converges to $M$ almost surely in the subtree size topology.
Further,  with probability one, $M$ has full support.

\vspace{.5mm}
\emph{(b)} Suppose that~\eqref{eq:reprSTS} holds 
for some possibly random $M\in\fM_\infty$, where $M$ has full support with 
probability one. Then $X$ is locally exchangeable.  
\end{theorem}

\begin{proof} (a) For each $u\in \bV\,$ de Finetti's theorem provides a possibly random
driving measure, here represented by a probability vector $p_u=(p_u(-1),p_u(0),p_u(1))$,
such that the sequence $Y(u)$ is conditionally i.i.d.\ with distribution $p_u$, 
which may be written as
\begin{equation*}
    \cL\bigl(Y(u)\big| p_u\bigr) = p_u^{\otimes \bN}\quad\text{for all }u\in \bV. 
\end{equation*}
By the convergence part of de Finetti's theorem,
\begin{equation*}
   \frac{1}{n}\Bigl|\bigl\{j\in [n]:\, Y_j(u)=k\bigr\}\Bigr|\; \to\; p_u(k)
                                         \quad\text{a.s.\ as }n\to\infty 
\end{equation*}
for all $u\in\bV$ and $k\in\{-1,0,1\}$. Clearly, 
\begin{equation*}
      \Bigl|\bigl\{v\in\bV:\, u+v\in X_{\tau_u+n}\bigr\}\Bigr|   \; = \;  
             1 + \Bigl|\bigl\{j\in [n]:\, Y_j(u)=-1\bigr\}\Bigr| +
                     \Bigl|\bigl\{j\in [n]:\, Y_j(u)=1\bigr\}\Bigr|  
\end{equation*}
for all $n\in\bN$. With $\psi(u):=p_u(-1)+p_u(1)$ we thus obtain
\begin{equation}\label{eq:dFsts}
    t(X_n,u)\to \psi(u) \quad\text{a.s.\ as }n\to\infty 
\end{equation}
for all $u\in\bV$. As $\psi(u)=\psi(u0)+\psi(u1)$ for all $u\in\bV\,$ 
the set function $M$ with $M(B_u)=\psi(u)$, $u\in\bV$, satisfies 
condition~\eqref{eq:finadd} in 
Lemma~\ref{lem:cara}. Together with $M(\bV) =\psi(\emptyset)=1$ this provides a
(unique)
$M\in\fM_\infty$, and~\eqref{eq:dFsts} shows that $X_n$ converges a.s.\ in the 
subtree size topology to $M$ as $n\to\infty$.

For the proof of~\eqref{eq:reprSTS} we first argue that $(\mu,A)\mapsto\DST(\mu)(A)$
defines a probability kernel from $\fM_\infty$ to $\fM_\uparrow$, both endowed with the
measurable structure generated by the insertion maps. For each $\mu\in\fM_\infty$,
$A\mapsto \DST(\mu)(A)$ is obviously a probability measure on 
$(\bB_\uparrow, \cB_\uparrow)$. For the measurability of
$\mu\mapsto \DST(\mu)(A)$ we may 
take $A$ to be of the form $A=\{X_1=x_1,\ldots,X_k=x_k\}$ with some $k\in\bN$,
$x_i\in\bB_i$ for $i\in[k]$ and  $x_i\subset x_{i+1}$ for $i\in [k-1]$. The increasing 
trees are described by the nodes $v_i$ with $x_i=x_{i-1}\cup\{v_i\}$ for 
$i=2,\ldots,k$, and with these the algorithm leads to
\begin{equation}\label{eq:kernel1}
    \DST(\mu)(A) = \prod_{i=2}^k \mu(B_{v_i}).
\end{equation}
The right hand side of~\eqref{eq:kernel1} is a measurable function of $\mu$.

We now use that, conditionally on $M\equiv \mu$ 
for some fixed $\mu\in\fM_\infty$, each of the local counting processes $Y(u)$, 
$u\in\bV$, is simply a sequence of independent random variables with values in 
$\{-1,0,1\}$ and probability mass function $p_u(-1)=\mu(B_{u0})$,
$p_u(0)=1-\mu(B_{u0})-\mu(B_{u1})$ and $p_u(1)=\mu(B_{u1})$.
Let $\tilde X=(\tilde X_n)_{n\in\bN}\sim\DST(\mu)$. Clearly,
$X_1=\tilde X_1=\{\emptyset\}$. Further, $X_{n+1}=X_n\cup\{V_{n+1}\}$ with
$V_{n+1}\in\partial X_n$, and $\tilde X_{n+1}=\tilde X_n\cup\{\tilde V_{n+1}\}$ with
$\tilde V_{n+1}\in\partial \tilde X_n$. Hence~\eqref{eq:reprSTS} will follow by induction if, for all $n\in\bN$, $x_1\in\bB_1,\ldots,x_n\in\bB_n$ with $x_1\subset\cdots\subset x_n$,
and all $v\in\partial x_n$, 
\begin{equation}\label{eq:nextnode}
    P\bigl(V_{n+1}=v\big|X_1=x_1,\ldots,X_n=x_n\bigr)\; 
         =\; P\bigl(\tilde V_{n+1}=v\big|\tilde X_1=x_1,\ldots,\tilde X_n=x_n\bigr).
\end{equation}
For the proof we may assume that $v=u0\in\partial x$ with $u\in x$,
the argument for the other case $v=u1$ being similar. 
Then $V_{n+1}=v$ holds if and only if $Y_k(u)=-1$, where $k:=n-\tau_u+1$ is a function of
$x_1,\ldots,x_n$, so that the  left hand side of~\eqref{eq:nextnode} evaluates to
$p_u(-1)=\mu(B_v)$. Further, from  the definition of the DST
algorithm it follows that the right hand side of~\eqref{eq:nextnode} is equal to
$P(\xi_{n+1}\in B_v)$, where $(\xi_{n})_{n\in\bN}$ is the input sequence. As these have
distribution $\mu$, this is again equal to~$\mu(B_v)$.

In order to prove the support statement we first note that the 
representation~\eqref{eq:reprSTS} gives
\begin{equation}\label{eq:contra}
    P(X\in A) =\int_{\fM_\infty} \DST(\mu)(A)\, \cL(M)(d\mu)
\end{equation}
for all $A\in\cB_\uparrow$. With $A=\bigl\{\,\bigcup_{n=1}^\infty x_n=\bV\,\bigr\}$  
the assumption leads to the value 1 on the left hand side. 
Also, $0\le \DST(\mu)(A)\le 1$ for all $\mu\in\fM_\infty$.
Now suppose that $M(B_u)=0$ has positive probability for some $u\in\bV$, so that
\begin{equation*}
   \cL(M)\bigl(\{\mu\in\fM_\infty:\, \mu(B_u)=0\}\bigr) > 0.
\end{equation*}
For each $\mu$ in this set, 
\begin{equation*}
   \DST(\mu)\Bigl(\bigcup_{n=1}^\infty x_n\subset \bV\setminus B_u\}\Bigr)
               \; =\; 1.
\end{equation*}
by the definition of the DST algorithm. 
This means that the integrand on the right hand side of~\eqref{eq:contra} vanishes
on a set of positive probability for the integrating distribution, which implies that 
the integral is strictly smaller than~1.

\vspace{.5mm}
(b) If $\mu(B_u)>0$ for all $u\in\bV$ then it follows 
from the definition of the DST algorithm that
each of the local increment processes is a sequence
of independent and identically distributed random variables, with 
$P(Y_n(u)=-1)=\mu(B_{u0})$,  $P(Y_n(u)=1)=\mu(B_{u1})$, and 
$P(Y_n(u)=0)=1-\mu(B_{u0})-\mu(B_{u1})$ for all $u\in\bV$.
Hence a tree sequence with distribution $\DST(\mu)$, $\mu$ with full support, 
is locally exchangeable. To see that this property survives the mixing 
operation we recall that 
$Y(u)$ is a deterministic function $\Psi_u$ of $X$, so that~\eqref{eq:reprSTS} leads to
\begin{equation*}
  \cL\bigl(Y(u)\bigr) \;=\;  \cL(X)^{\Psi_u} 
                     \;=\; \Bigl(\int \DST(\mu)\, \cL(M)(d\mu)\Bigr)^{\Psi_u}
                     \;=\; \int \DST(\mu)^{\Psi_u}\, \cL(M)(d\mu).
\end{equation*}
It follows from the above argument for the DST case that only distributions of 
i.i.d.\ sequences appear inside the integral, hence $Y(u)$ is exchangeable.
\end{proof}

It may seem surprising that only local conditions on the processes $Y(u)$, $u\in\bV$, 
are needed. However, the general structure of the tree sequence leads to deterministic
relations between these. For example, let $B_{u(1)},\ldots,B_{u(d)}$ be a partition 
of $\bV_\infty$ (or, equivalently, 
$\{u(j):\,j\in [d]\}=\partial x$ for some $x\in\bB$) and let
$\rho=\max_{j\in[d]}\tau_{u(j)}$. Then the individual local 
increment processes can be combined into a $d$-dimensional process 
$Y^0=(Y^0_1,\ldots,Y^0_d)$ with values in $\{-1,0,1\}^d$ by
\begin{equation*}
    Y^0_{j,n} \,=\, Y_{n+\rho-\tau_{u(j)}}(u(j))\quad\text{for all }j\in[d],\, n\in\bN.
\end{equation*}
With $e_j$ the $j$th canonical basis vector
of $\bR^d$ it then holds that, for all $n\in\bN$,  
$Y^0_n$ is equal to $e_j$ or $-e_j$ for some $j\in [d]$.
Similarly, for random $M$, the individual driving random vectors $p_u$, $u\in\bV$,
need not be independent, as evidenced by $M_{\text{\tiny BST}}$.

As a geometric consequence of Theorem~\ref{thm:charSTS} we obtain that the 
set of distributions of locally exchangeable tree sequences is convex, in fact
in affine-linear and one-to-one correspondence with the full support subset 
of $\fM_\infty$. This in turn can be used to identify its extremal elements
as $\DST(\mu)$, where $\mu$ has support $\bV_\infty$. In particular, 
the BST distribution is locally exchangeable, with random driving measure 
$M=M_{\text{\tiny BST}}$ given in Example~\ref{ex:sts3}.

\section{A second order result}\label{sec:FCLT}

In Section~\ref{sec:sts} we examined a specific notion of convergence for general
sequences in $\bB$ and in Section~\ref{sec:symm} we found a representation for
a class of increasing random trees where this notion appears. The subtree size 
convergence in Theorem~\ref{thm:charSTS} may be interpreted  as a strong law of 
large numbers, with a possibly random limit. 
For such sequences it makes sense to consider an 
analogue of the central limit theorem. As all distributions of locally exchangeable
sequences arise as mixtures of $\DST(\mu)$, $\mu\in\fM_\infty$, we are thus lead to
consider sequences $X=(X_n)_{n\in\bN}\sim\DST(\mu)$, with $\mu$ of full support. 
Our aim is a functional central limit theorem for the stochastic processes
\begin{equation}\label{eq:defZn}
  Z_n:=\sqrt{n}\bigl(t(X_n,u)-\mu(B_u)\bigr)_{u\in\bV}\,, \quad n\in\bN,
\end{equation} 
meaning that $Z_n$ converges in distribution to a Gaussian process 
$Z=(Z_u)_{u\in\bV}$ as $n\to\infty$. For a general locally exchangeable 
sequence we then obtain asymptotic mixed normality by conditioning on the limit 
$M$ in Theorem~\ref{thm:charSTS}. 

The distributional convergence is based on an infinite-dimensional space $\bL$ of
functions on~$\bV$ that contains the range of $Z_n$ with probability one. For 
general  subtree size convergence we may take $\bL$ to be the vector space 
$\bR^{\bV}$ of all real functions on $\bV$, together with the product topology, 
i.e.~of convergence of coordinates. 
For processes $X=(X_t)_{t\in T}$ with time parameter $t\in T$ it is customary to denote
the distributions of random vectors $(X_{t_1},\ldots,X_{t_k})$, $k\in\bN$ and 
$t_1,\ldots,t_k\in T$, as the \emph{finite-dimensional distributions} of $X$.

\begin{theorem}\label{thm:fidis}
Let $(X_n)_{n\in\bN}$ be a sequence of random binary trees with distribution
$\DST(\mu)$ where $\mu$ has support $\bV_\infty$.
Then there exists
a centered Gaussian process $Z=(Z_u)_{u\in\bV}$ with covariance function
\begin{equation}\label{eq:covZ}
 \cov(Z_u,Z_v)\; = \; \begin{cases} \mu(B_u)(1-\mu(B_u)), &\text{if } u=v,\\
                              \mu(B_v)(1-\mu(B_u)),  &\text{if } u\prec v,\\
                              \mu(B_u)(1-\mu(B_v)),  &\text{if } v\prec u,\\
                              -\mu(B_u)\mu(B_v), &\text{else,}
                      \end{cases}
\end{equation}
and with this process it holds that
\begin{equation}\label{eq:CLT}
    \sqrt{n}\bigl(t(X_n,u)-\mu(B_u)\bigr)_{u\in\bV} \, \todistr \, Z 
                   \quad\text{as }n\to\infty,
\end{equation} 
where the convergence in distribution refers to $\bR^{\bV}$ endowed with the
product topology.
\end{theorem}

\begin{proof}
Suppose that $A\subset\bV$ is finite and such that the sets 
$B_u$, $u\in A$, are pairwise disjoint with $\sum_{u\in A}\mu(B_u)=1$. 
We know from the proof of Theorem~\ref{thm:sts}\,(b) that 
$\tau_u= \inf\{n\in\bN:\, u\in X_n\}$ is finite with probability one, 
for all $u\in\bV$. Let $\rho:=\sup\{\tau_u:\, u\in A\}$ and fix some $k\in\bN$. 
Then it follows from the description of the DST algorithm with input 
$(\xi_n)_{n\in\bN}$ that,
conditionally on $\rho\le k$, the random vector $Y_n=(Y_{n,u})_{u\in A}$ with components
\begin{equation}\label{eq:defY}
        Y_{n,u} \,:=\, \bigl|\{k<m\le n:\, \xi_m\in B_u\}\bigr| 
                   \,=\, \bigl|\sigma(X_n,u)\bigr| - \bigl|\sigma(X_k,u)\bigr|, 
                      \quad u\in A,
\end{equation}
has  a multinomial distribution, with parameters $n-k$ (for the number of trials) and 
$(\mu(B_u))_{u\in A}$ (for the vector of success probabilities). By the central limit theorem
for these distributions,
\begin{equation}\label{eq:multCLT}
  \sqrt{n-k}\,\Bigl(\frac{1}{n-k} Y_{n,u}-\mu(B_u)\Bigr)_{u\in A}\, 
          \todistr\, Z=(Z_u)_{u\in A}\quad \text{as }n\to\infty,
\end{equation} 
where the random vector $Z$ is centered normal with covariances
\begin{equation*}
     \cov(Z_u,Z_v) \, = \, \begin{cases} \mu(B_u)(1-\mu(B_u)), &\text{if } u=v,  \\
                                         -\mu(B_u)\mu(B_v)\, &\text{if } u\not=v.
                           \end{cases}
\end{equation*} 
This implies 
\begin{equation}\label{eq:fidiCLT}
  \sqrt{n}\bigl(t(X_n,u)-\mu(B_u)\bigr)_{u\in A}\, 
          \todistr\, Z=(Z_u)_{u\in A}\quad \text{as }n\to\infty,
\end{equation} 
as the difference between the left hand sides in~\eqref{eq:multCLT} and~\eqref{eq:fidiCLT}
converges to zero with probability one as $n\to\infty$ because of~\eqref{eq:defY}. All this
is conditionally on $\rho\le k$ for some $k\in\bN$. However, as $k$ does not appear 
in~\eqref{eq:fidiCLT} and as $\rho<\infty$ with probability one, the last statement 
even holds unconditionally.

On this basis we now deduce the convergence of the finite-dimensional 
distributions together with the covariance function of the limit process. 

For $k\in\bN$ fixed we have asymptotic normality of the random vector $Z(k)$ 
associated with $A:=\bV_k$ from the above argument. This yields 
joint asymptotic normality for all $Z_u$ with $|u|\le k$, $u\in\bV$, as these 
variables are all linear functions of the vector $Z(k)$. 

Now let $u,v\in\bV$. If $u\prec v$ and $|v|=k$ then, due to the asymptotic 
negligibility of the difference, we may use the partition of $B_u$ into sets $B_w$ with
$|w|=k$ to obtain, with $A(u,v):=\{w\in\bV_k:\, w\not=v, u\prec w\}$ 
\begin{equation*}
\cov(Z_v,Z_u) \ =\  \var(Z_v) + \cov\Bigl(Z_v,\sum_{w\in A(u,v)} Z_w\Bigr)
              \ =\  \mu(B_v)(1-\mu(B_u)).
\end{equation*}    
The case $v\prec u$ follows by symmetry. Finally, if neither $u\prec v$ nor $v\prec u$
then $\{B_u,B_v\}$ can be augmented to a system $\cB_A$ to which the first step applies. 

As explained in \cite[p17]{Bill}, for $\bR^\bV$ this already implies the asserted convergence 
in distribution, together with the existence of the Gaussian process $Z$.
\end{proof}

The following is now an immediate consequence of the theorem and 
the mixture representation of the BST distribution, as
$M_{\text{\tiny BST}}$ has support $\bV_\infty$ with probability one.

\begin{corollary}
Let $X=(X_n)_{n\in\bN}$ be the sequence of random trees generated by the BST algorithm 
with independent random variables uniformly distributed on the unit interval. Let $Z(\mu)$
be the centered Gaussian process with covariance function given in~\eqref{eq:covZ}, and let
$M_{\text{\tiny\rm BST}}$ be as defined in~\eqref{eq:Mbst}.
Then
\begin{equation}\label{eq:BSTCLT}
    \sqrt{n}\bigl(t(X_n,u)-M_{\text{\tiny\rm BST}}(B_u)\bigr)_{u\in\bV} \, \todistr \, Z 
                   \quad\text{as }n\to\infty,
\end{equation} 
where the distribution of $Z$ is given by the distribution of $M_{\text{\tiny\rm BST}}$  and 
the conditional distribution $\cL(Z|M_{\text{\tiny\rm BST}}=\mu)=\cL(Z(\mu))$, $\mu\in\fM_\infty$. 
\end{corollary}

\acknowledgements
\label{sec:ack} 
The referees' comments have led to a significant improvement of the paper.

\nocite{*}
\bibliographystyle{abbrvnat}

\begin{thebibliography}{20}
\providecommand{\natexlab}[1]{#1}
\providecommand{\url}[1]{\texttt{#1}}
\expandafter\ifx\csname urlstyle\endcsname\relax
  \providecommand{\doi}[1]{doi: #1}\else
  \providecommand{\doi}{doi: \begingroup \urlstyle{rm}\Url}\fi

\bibitem[Billingsley(1968)]{Bill}
P.~Billingsley.
\newblock \emph{Convergence of probability measures}.
\newblock Wiley, New York, 1968.

\bibitem[Devroye(1998)]{DevSplit}
L.~Devroye.
\newblock Universal limit laws for depths in random trees.
\newblock \emph{SIAM J. Comput.}, 28:\penalty0 409--432, 1998.

\bibitem[Diaconis and Janson(2008)]{DiaconisJanson}
P.~Diaconis and S.~Janson.
\newblock Graph limits and exchangeable random graphs.
\newblock \emph{Rend. Mat. Appl.}, 28:\penalty0 33--61, 2008.

\bibitem[Drmota(2009)]{Drmota}
M.~Drmota.
\newblock \emph{Random trees. An interplay between combinatorics and
  probability}.
\newblock Springer, Wien, 2009.

\bibitem[Elek and Tardos(2022)]{ElekTardos}
G.~Elek and G.~Tardos.
\newblock Convergence and limits of finite trees.
\newblock \emph{Combinatorica}, 42:\penalty0 821--852, 2022.

\bibitem[Evans et~al.(2012)Evans, Gr{\"u}bel, and Wakolbinger]{EGW1}
S.~Evans, R.~Gr{\"u}bel, and A.~Wakolbinger.
\newblock Trickle-down processes and their boundaries.
\newblock \emph{Electron. J. Probab.}, 17:\penalty0 58pp., 2012.

\bibitem[Evans et~al.(2017)Evans, Gr{\"u}bel, and Wakolbinger]{EGW2}
S.~Evans, R.~Gr{\"u}bel, and A.~Wakolbinger.
\newblock Doob-{M}artin boundary of {R}{\'e}my's tree growth chain.
\newblock \emph{Ann. Probab.}, 45:\penalty0 225--277, 2017.

\bibitem[Gr{\"u}bel(2013)]{GrSemBerKMK}
R.~Gr{\"u}bel.
\newblock Kombinatorische {M}arkov-{K}etten.
\newblock \emph{Math. Semesterber.}, 60:\penalty0 185--215, 2013.

\bibitem[Gr{\"u}bel(2015)]{GrueDMTCS}
R.~Gr{\"u}bel.
\newblock Persisting randomness in randomly growing discrete structures: graphs
  and search trees.
\newblock \emph{Discrete Math. Theor. Comp. Sci.}, 18:\penalty0 23pp., 2015.

\bibitem[Gr{\"u}bel(2023+)]{Grue22}
R.~Gr{\"u}bel.
\newblock Ranks, copulas, and permutons.
\newblock \emph{Metrika}, to appear, 2023+.

\bibitem[Hoppen et~al.(2013)Hoppen, Kohayakawa, Moreira, R\'ath, and
  Sampaio]{Hopp}
C.~Hoppen, Y.~Kohayakawa, C.~Moreira, B.~R\'ath, and R.~Sampaio.
\newblock Limits of permutation sequences.
\newblock \emph{J. Comb. Theory, Ser. B}, 103:\penalty0 93--113, 2013.

\bibitem[Janson(2011)]{JansonPoset}
S.~Janson.
\newblock Poset limits and exchangeable random posets.
\newblock \emph{Combinatorica}, 31:\penalty0 529--563, 2011.

\bibitem[Janson(2012{\natexlab{a}})]{JansonSurvey}
S.~Janson.
\newblock Simply generated trees, conditioned {G}alton-{W}atson trees, random
  allocations and condensation.
\newblock \emph{Probab. Surv.}, 9:\penalty0 103--252, 2012{\natexlab{a}}.

\bibitem[Janson(2012{\natexlab{b}})]{JansonTrees}
S.~Janson.
\newblock Tree limits and limits of random trees.
\newblock \emph{Combin. Probab. Comput.}, 30:\penalty0 849--893,
  2012{\natexlab{b}}.

\bibitem[Kallenberg(1997)]{KallFMP}
O.~Kallenberg.
\newblock \emph{Foundations of modern probability}.
\newblock Springer, New York, 1997.

\bibitem[Kallenberg(2005)]{KallSym}
O.~Kallenberg.
\newblock \emph{Probabilistic symmetries and invariance principles}.
\newblock Springer, New York, 2005.

\bibitem[Knuth(1973)]{Knuth}
D.~E. Knuth.
\newblock \emph{The art of computer programming 3. Sorting and searching}.
\newblock Addison–Wesley, Reading, MA, 1973.

\bibitem[Lov\'asz(2012)]{Lovasz}
L.~Lov\'asz.
\newblock \emph{Large networks and graph limits}.
\newblock American Mathematical Society Colloquium Publications 60. Amer. Math.
  Soc., Providence, RI, 2012.

\bibitem[Lov\'asz and Szegedy(2006)]{LovSze}
L.~Lov\'asz and B.~Szegedy.
\newblock Limits of dense graph sequences.
\newblock \emph{J. Combin. Theory, Ser. B}, 96:\penalty0 933--957, 2006.

\bibitem[R{\'e}my(1985)]{Remy}
J.-L. R{\'e}my.
\newblock Un procédé itératif de dénombrement d'arbres binaires et son
  application à leur génération aléatoire.
\newblock \emph{RAIRO Inform. Théor.}, 19:\penalty0 179--195, 1985.

\end{thebibliography}
\label{sec:biblio}

\end{document}